\documentclass[10pt,reqno]{amsart}
\usepackage[english]{babel}
\usepackage{amsmath}
\usepackage{amssymb}
\usepackage{amsfonts}
\usepackage{graphicx}
\usepackage{xcolor}
\usepackage[colorlinks]{hyperref}

\newcommand\orcidicon[1]{\href{https://orcid.org/#1}{\includegraphics[scale=0.02]{orcid.pdf}}}

\newtheorem{lemma}{Lemma}
\newtheorem{theorem}{Theorem}

\newtheorem{corollary}{Corollary}

\definecolor{myblue}{rgb}{0.03, 0.27, 0.79}

\hypersetup{linkcolor=blue,filecolor=black,urlcolor=blue, citecolor=blue}

\begin{document}
	\renewcommand{\refname}{References}
	
	\thispagestyle{empty}
	
	\title[\sc Central orders in simple superalgebras]{\large Central orders in simple right-alternative superalgebras and right-symmetric algebras.}
	
	\author[A.S. PANASENKO]{{\bf \large \textsc{A.S. Panasenko}}\\ \\
	}
	\address{Alexander Sergeevich Panasenko 
		\newline\hphantom{iii} Sobolev Institute of Mathematics,
		\newline\hphantom{iii} pr. Koptyuga, 4,
		\newline\hphantom{iii} 630090, Novosibirsk, Russia}%
	\email{\textcolor{blue}{a.panasenko@g.nsu.ru}}%
	
	\thanks{\sc Panasenko, A.S.,
		Central orders in simple right-alternative superalgebras and right-symmetric algebras}
	\thanks{\copyright \ 2024 Panasenko A.S.}
	\thanks{\rm The work was carried out within the framework of the state assignment of the SB RAS Institute of Mathematics, topic FWNF-2022-0002.}
	
	\maketitle {\small
		\begin{quote}
			\noindent{\bf Abstract:} We consider some recently constructed examples of simple finite-dimensional right-alternative superalgebras and right-symmetric algebras. We prove that the central order in any of these algebras and superalgebras is embedded in a finite module over its center (or over the even part of its center in the case of superalgebras).\medskip
			
		\end{quote}
	}
	
	\bigskip
	
	\section{Introduction}
	The associator is the most important multilinear function in the theory of non-associative algebras. Its first significant application is in the theory of alternative algebras. By definition, an alternative algebra is an algebra in which the associator is a skew-symmetric function of its arguments. The opposite notion is an asymmetric algebra, in which the associator is a symmetric function of its arguments. As E.~Kleinfeld showed in \cite{Kleinfeld}, these algebras are too close to associative: in particular, any prime asymmetric algebra over a field of characteristic neither 2 nor 3 is associative.
	
	For decades, attempts have been made to study the theory of algebras in which the associator has weaker (skew)symmetry properties. Algebras in which the associator is a symmetric function of the second and third arguments are called right-symmetric. These objects (and the left-symmetric algebras anti-isomorphic to them) have arisen in various areas of abstract algebra and geometry: for the classification of convex homogeneous cones \cite{Vinberg}, in the study of the Yang-Baxter equation \cite{Yang}, and many other applications are given in the survey \cite{Burde}. Unfortunately, the variety of these algebras is so large that it leaves no chance for an adequate structure theory. Many classical results for other varieties turn out to be incorrect for right-symmetric algebras even in the finite-dimensional case. A special case of left-symmetric algebras --- Novikov algebras --- has an acceptable (though far from ideal) finite-dimensional structure theory with hopes of transition to infinite-dimensional algebras (see the papers \cite{Xu}, \cite{Zhelyabin1}, \cite{Zhelyabin2}, \cite{P2022}, \cite{P2024}).
	
	Algebras in which the associator is a skew-symmetric function of the second and third arguments are called right-alternative. The finite-dimensional theory of right-alternative algebras is very close to the theory of alternative algebras. As A. Albert showed in \cite{Albert}, every simple finite-dimensional right-alternative algebra is alternative. In the infinite-dimensional case, it is also evident that the presence of non-alternative right-alternative algebras can be regarded as an error. For example, in \cite{Skor} it is shown that every non-degenerate (in particular, unital simple) right alternative algebra is alternative. At the same time (\cite{Miheev}), there exist simple right alternative nil-algebras of index 3.
	
	However, the transition to superalgebras greatly expands the possibilities. Initially, the study of right alternative superalgebras was carried out within the framework of special cases -- alternative and (-1,1) superalgebras. In the last decade, the theory of simple finite-dimensional right alternative superalgebras has been actively developed mainly in the works of S.~V.~Pchelintsev and O.~V.~Shashkov (a brief presentation is available in the review \cite{PS0}).
	
	In the work of E.~Formanek \cite{Formanek} it was proved that the central order in a finite-dimensional central simple associative algebra (i.e. in a matrix algebra) is embedded in a finitely generated module over its center. In \cite{P2015} and \cite{P2019} this result was extended to alternative and Jordan algebras. In \cite{P2020} simple alternative and Jordan superalgebras were considered, and for most of them it was proved that the central order in these superalgebras can be embedded in a finitely generated module over the even part of its center.
	
	This paper is a natural continuation of the above-mentioned articles. In this paper we consider central orders in some important simple right-symmetric algebras and simple right-alternative superalgebras constructed in recent years. For these examples it is proved that they can be embedded in a finite module over their center (the even part of their center in the case of superalgebras).
	
	\section{Examples of simple right-symmetric algebras}
	
	We will use the standard notations for the associator and commutator in an arbitrary algebra: $(x,y,z) = (xy)z - x(yz)$, $[x,y]=xy-yx$.
	
	\textbf{Definition.} Let $A$ be an algebra and for any $x,y,z\in A$ we have $(x,y,z)=(x,z,y)$. Then the algebra $A$ is called a \textit{\textbf{right-symmetric algebra}}.
	
	\smallskip One way to construct new right-symmetric algebras is to construct an endomorph. This method, in particular, allowed us to construct the following example of the so-called \textbf{\textit{matrix RS-algebra}}, which first appeared in \cite{PozhShest2}.
	
	\medskip\textbf{Example 1.} \textit{Let $F^n$ be the direct sum of $n$ copies of the field $F$. In the algebra $\overline{F^n}$ there is a basis $e_1,\dots,e_n$, where $e_i=(0,\dots,0,1,0,\dots,0)$ and the unit is at the $i$-th place. Consider the linear mapping $\overline{\phantom{a}}: F^n \to M_n(F)$, $\overline{e_i} = e_{ii}$. Let us define an algebra structure on the direct sum of subalgebras $F^n + M_n(F)$ according to the rule
		\[v \cdot A = vA, \quad A\cdot v = vA + [A,\overline{v}].\]
		If $n>1$, then $F^n + M_n(F)$ is a simple non-associative right-symmetric algebra.
	}
	
	Let $V_2$ be the standard module of rows over $M_2(F)$. Let us fix a mapping $\pi: V_2 \to M_2(F)$ by the rule
	\[\pi(x,y) = \begin{pmatrix}
		x & y \\
		x & y
	\end{pmatrix}.\]
	For a fixed matrix $C\in M_2(F)$ we define $\psi_{C}(x) = \pi(x) \odot C$, where $\odot$ is the Hadamard product.
	
	\medskip\textbf{Example 2.} \textit{Let $V_2$ be the space of rows of length 2 over the field $F$ and let $V_2$ be given some bilinear operation $x\bullet y$. On the direct sum of spaces $\mathcal{A} = V_2 + M_2(F)$ we define multiplication, assuming that $M_2(F)$~--- is a matrix subalgebra, $vA$ is defined as the multiplication of a row $v$ by a matrix $A$, and there exists $C\in M_2(F)$ such that
		\[Av = vA + [A,[R^{\bullet}_{v}]], \qquad wu = w\bullet u + \psi_C(u)\pi(w)\]
		for any $v,w,u\in V_2$, $A\in M_2(F)$, where $[\varphi]$ is the matrix of a linear transformation of the space $V_2$.
	}
	
	In the paper \cite{PozhShest1} A.~P.~Pozhidaev and I.~P.~Shestakov proved that a right-symmetric algebra of the form $W + M_2(F)$, where $W$ is an irreducible module over $sl_2$, has the form $\mathcal{A}$ from Example~2.
	
	\section{Central orders in simple right-symmetric algebras}
	
	\textbf{Definition.} Let $A$ be an algebra. Then the \textit{\textbf{center}} of $A$ is the following set:
	\[Z(A) = \{z\in A\mid (z,x,y)=(x,z,y)=(x,y,z)=[x,z]=0 \quad \forall x,y\in A\}.\]
	
	Let $A$ be an algebra, $Z=Z(A)$ be the center of $A$, and $Z$ contain no zero divisors of the entire algebra $A$. Consider the set of pairs $S=\{(a,z)\mid a\in A, z\in Z, z\neq 0\}$. On the set $S$ we consider the equivalence relation:
	\[(a,x)\sim (b,y) \longleftrightarrow ay=bx.\]
	We denote the set of equivalence classes by $Z^{-1}A$, and the equivalence class containing the element $(a,z)$ will be denoted by $\frac{a}{z}$. On $Z^{-1}A$ we define addition and multiplication as on ordinary fractions. Then $Z^{-1}A$ is an algebra containing $A$ as a subalgebra. In addition, the algebra $Z^{-1}A$ can be considered over the field of fractions $Z^{-1}Z$.
	
	\medskip\textbf{Definition.} Let $A$ be an algebra, $Z=Z(A)$ be the center of $A$, and $Z$ does not contain zero divisors of the entire algebra $A$. Then the algebra $A$ is called a \textbf{\textit{central order}} in the algebra $Z^{-1}A$.
	
	\medskip Note that $A\cap Z^{-1}Z = Z$. Indeed, the inclusion $Z\subseteq A\cap Z^{-1}Z$ is obvious. Let $b\in A\cap Z^{-1}Z$. Since $Z^{-1}Z = Z(Z^{-1}A)$ (proved in \cite{Zhevl}, Proposition~8.2), it follows that $[b,x]=(b,x,y)=(x,b,y)=(x,y,b)=0$ for any $x,y\in A$. But $b\in A$, so that $b\in Z(A)$.
	
	\medskip There is an explicit multiplication on the matrix RS-algebra, so we begin our study of central orders in right-symmetric algebras with this case.
	
	\begin{lemma}
		Let $B$ be an algebra that is a central order in the matrix RS-algebra $A = F^n + M_n(F)$. Then $B$ can be embedded into a finitely generated $Z(B)$-module.
	\end{lemma}
	
	\begin{proof}
		Denote $\overline{B} = Z^{-1}B$. Note that $F = Z^{-1}Z$. The algebra $\overline{B}$ has a basis $e_1,\dots,e_n, e_{ij}$ over the field $F$, where $e_i=(0,\dots,0,1,0,\dots,0)$ and the unit is at the $i$-th place. There is $z\in Z$ such that $e_i = \frac{f_i}{z}$, $e_{ij} = \frac{f_{ij}}{z}$ and $f_i,f_{ij}\in B$. Let $a=(\alpha_1,\dots,\alpha_n) + \sum\beta_{ij}e_{ij} \in B$, $\alpha_i,\beta_{ij}\in F$. If $F\neq m$, then
		\[
		\sum\limits_{i=1}^n (f_{im}(f_{mm}(af_{kk})))f_{ki} = \beta_{mk} z^4.
		\]
		If $k=m$, then for $s\neq k$
		\[\sum\limits_{i=1}^n (f_{ik}(f_{kk}((f_{kk}(af_{kk}))f_{ks})))f_{si} = \beta_{kk}z^6.\]
		In addition, for $s\neq k$ we have
		\[\sum\limits_{i=1}^n (f_{ik}(f_{km}(f_{mm}((f_{kk}(af_{kk}))f_{km}))))f_{mi} = \alpha_k z^7.\]
		Thus, $\alpha_k z^7, \beta_{km} z^6 \in B\cap F= Z$, whence
		\[az^7 \in \sum\limits_{i=1}^n Z e_i + \sum\limits_{i,j=1}^n Ze_{ij}.\] Thus, the $Z$-module $Bz^7$ is embedded in a finitely generated $Z$-module. It remains to note that the $Z$-module $Bz^7$ is isomorphic to the $Z$-module~$B$.
	\end{proof}
	
	Let us pass to another example of a simple right-symmetric algebra.
	
	\begin{lemma}
		Let $B$ be an algebra which is a central order in the algebra $\mathcal{A} = V_2 + M_2(F)$, where the spaces $F(1,0)$ and $F(0,1)$ are not closed under the operation $\bullet$. Then $B$ can be embedded in a finitely generated $Z(B)$-module.
		
	\end{lemma}
	
	\begin{proof}
		Denote $\overline{B} = Z^{-1}B$. Note that $F = Z^{-1}Z$. The algebra $\overline{B}$ has a basis $e_1 = (1,0), e_2 = (0,1)$, $e_{11},e_{12},e_{21},e_{22}$ over the field $F$. Let the bilinear operation $\bullet$ on $V_2$ be defined as follows:
		\[e_1\bullet e_1 = (\gamma_1, \delta_1), e_1\bullet e_2 = (\gamma_2,\delta_2), e_2\bullet e_1 = (\gamma_3,\delta_3), e_2\bullet e_2 = (\gamma_4,\delta_4).\]
		Let $C=\begin{pmatrix}
			\varepsilon_1 & \varepsilon_2 \\
			\varepsilon_3 & \varepsilon_4
		\end{pmatrix}$.
		There is $z\in Z$ such that $e_i = \frac{f_i}{z}$, $e_{ij} = \frac{f_{ij}}{z}$, $\gamma_i = \frac{\mu_i}{z}$, $\delta_i = \frac{\nu_i}{z}$, $\varepsilon_i = \frac{\xi_i}{z}$ and $f_i,f_{ij}\in B$, $\mu_i,\nu_i, \xi_i \in Z$. Let $a = (\alpha,\beta) + \sum\alpha_{ij} e_{ij}$.
		
		1) Suppose $\delta_1\neq 0$. Then
		\begin{gather*}
			\chi_1(a)=-\frac{1}{\nu_1^3}(f_{21}(((f_1(af_{11}))f_{22})f_{21}))(\nu_1 f_{11} - \nu_1 f_{22} + (\nu_3 - \mu_1)f_{21}) = \alpha z^4,\\
			\chi_2(a)=-\frac{1}{\nu_1^3}(f_{21}(((f_1((af_{22})f_{21}))f_{22})f_{21}))(\nu_1 (f_{11} - f_{22}) + (\nu_3 - \mu_1)f_{21}) = \beta z^5.
		\end{gather*}
		But then
		\[f_{1i}(az^6-\chi_1(a)f_1z-\chi_2(a)f_2)f_{j1} + f_{2i}((az^6-\chi_1(a)f_1z-\chi_2(a)f_2))f_{j2} = \alpha_{ij}z^8.\]
		Thus, $\alpha z^4, \beta z^5, \alpha_{ij}z^8 \in B\cap F = Z$, whence $az^8\in Ze_1 + Ze_2 + \sum Ze_{ij}$. Thus, the $Z$-module $Bz^8$ is embedded in a finitely generated $Z$-module. It remains to note that the $Z$-module $Bz^8$ is isomorphic to the $Z$-module $B$.
		
		2) Suppose $\gamma_4 \neq 0$. Then
		
		\begin{gather*}
			\frac{1}{\mu_4^3}(f_{12}(((f_2(af_{22}))f_{11})f_{12}))(-\mu_4 f_{11} + \mu_4 f_{22} + (\mu_2-\nu_4)f_{12}) = \beta z^4,\\
			\frac{1}{\mu_4^3}(f_{12}(((f_2((af_{11})f_{12}))f_{11})f_{12}))(-\mu_4 f_{11} + \mu_4 f_{22} + (\mu_2-\nu_4)f_{12}) = \alpha z^5.
		\end{gather*}
		Then, as in first case, we obtain $az^8 \in Ze_1 + Ze_2 + \sum Ze_{ij}$. Thus, the $Z$-module $Bz^8$ is embedded in a finitely generated $Z$-module. It remains to note that the $Z$-module $Bz^8$ is isomorphic to the $Z$-module $B$.
		
	\end{proof}
	
	Thus, we have proved the following theorem.
	
	\begin{theorem}
		Let $B$ be a right-symmetric algebra that is a central order either in the matrix RS-algebra or in the algebra $V_2+M_2(F)$, and the spaces $F(1,0)$ and $F(0,1)$ are not closed under the operation $\bullet$. Then $B$ is embedded in a finitely generated $Z(B)$-module.
	\end{theorem}
	
	\section{Examples of simple right-alternative superalgebras}
	
	\smallskip\textbf{Definition.} A $\mathbb{Z}_2$-graded algebra $B=A+M$, $AM+MA\subseteq M$, $A^2+M^2\subseteq A$, is called a \textbf{\textit{right-alternative superalgebra}} if for any homogeneous elements $x,y,z\in A\cup M$ we have
	\[(x,y,z) + (-1)^{|z||y|}(x,z,y) = 0,\]
	where $|a| = 0$ if $a\in A$ and $|a| = 1$ if $a\in M$.
	
	\smallskip A superalgebra $B=A+M$ is called a \textbf{\textit{a superalgebra of abelian type}} if the even part $A$ is associative and commutative, and the odd part $M$ is an associative $A$-bimodule. The following two right alternative superalgebras of abelian type are well known.
	
	\medskip\textbf{Example 3.} \textit{Let $F^n$ be a direct sum of $n$ copies of $F$ and $[F^n]$ be a vector space isomorphic to $F^n$. Consider $B_{n|n} = F^n + [F^n]$ --- the direct sum of vector spaces with the following multiplication:
		\begin{gather*}
			x \cdot y = xy,\\
			[x] \cdot [y] = xy,\\
			x \cdot [y] = [xy],\\
			[x]\cdot y = [x \tau(y)],
		\end{gather*}
		where $x,y\in F^n$, $xy$ is the product in $F^n$, $\tau(a_1,\dots,a_n) = (a_2,\dots, a_n, a_1)$. Then $B_{n|n}$ is a simple non-associative right-alternative superalgebra of abelian type.
	}
	
	\medskip\textbf{Example 4.} \textit{Let $[F^2]$ be a vector space isomorphic to $F^2$. Consider $B_{2|2}(\nu) = F^2 + [F^2]$ --- the direct sum of vector spaces with the following multiplication:
		\begin{gather*}
			x \cdot y = xy, \\
			[x] \cdot [y] = x\chi(y),\\
			x \cdot [y] = [xy],\\
			[x] \cdot y = [xy^{*}],
		\end{gather*}
		where $x,y\in F^2$, $xy$ is the product in $F^2$, $(a_1,a_2)^* = (a_2,a_1)$, $\chi(a_1,a_2) = (a_1+a_2, a_1+\nu a_2)$, $\nu\in F$. Then $B_{2|2}(\nu)$ is a simple non-associative right-alternative superalgebra of abelian type.
	}
	
	\medskip In the paper \cite{PS1} S.~V.~Pchelintsev and O.~V.~Shashkov proved that over an algebraically closed field of characteristic~0 any central simple finite-dimensional right-alternative superalgebra of abelian type is either associative, or isomorphic to the superalgebra $B_{n|n}$, or isomorphic to the superalgebra $B_{2|2}(\nu)$ for some $\nu$ in the ground field.
	
	\smallskip In the paper of J.~P.~Da~Silva, L.~S.~I.~Murakami and I.~P.~Shestakov \cite{Shest1} right-alternative superalgebras arose whose even part is an algebra of second-order matrices. Such algebras are called asymmetric doubles.
	
	\medskip\textbf{Example 5.} \textit{Let $w\in M_2(F)$ be a fixed matrix with nonzero trace and $[M_2(F)]$ be a vector space isomorphic to $M_2(F)$. Consider $B_{4|4}(w) = M_2(F) + [M_2(F)]$ --- the direct sum of vector spaces with the following multiplication:
		\begin{gather*}
			x \cdot y = xy,\\
			[x] \cdot [y] = \frac{\mathrm{tr}(y)}{\mathrm{tr}(w)} x,\\
			[x]\cdot y = [x\overline{y}],\\
			x \cdot [y] = [xy - ([x][y])^{D}],
		\end{gather*}
		where $x,y\in M_2(F)$, $xy$ is the product in $M_2(F)$, $a^{D} = aw - wa$ for any $a\in M_2(F)$, and $\overline{a}$ is the symplectic involution applied to an element $a\in M_2(F)$. Then $B_{4|4}(w)$ is a simple non-associative right-alternative superalgebra.
	}
	
	\medskip In the paper \cite{PS2} S.~V.~Pchelintsev and O.~V.~Shashkov proved that any asymmetric double over a field of characteristic not 2 is either alternative or isomorphic to the superalgebra $B_{4|4}(w)$ for some $w\in M_2(F)$, $\mathrm{tr}(w)\neq 0$.
	
	\section{Central orders in simple right-alternative superalgebras}
	
	Let $B = A + M$ be a $Z_2$-graded algebra and $Z=Z(B)_0$ be the even part of the center of $B$, where $Z$ does not contain any zero divisors of the entire algebra $B$. Then, as in the case of ordinary algebras, we can define a $Z_2$-graded algebra $Z^{-1}B$.
	
	As before, we will call the $Z_2$-graded algebra $B$ \textbf{\textit{central order}} in the $Z_2$-graded algebra $Z^{-1}B$. The algebra $Z^{-1}B$ can be considered over the field of fractions $Z^{-1}Z$.
	
	\medskip Let us begin our study with central orders in superalgebras of abelian type.
	
	\begin{lemma}
		Let $B$ be a right-alternative superalgebra which is a central order in the superalgebra $B_{n|n}$ and $Z = Z(B)_0$. Then $B$ can be embedded into a finitely generated $Z$-module.
	\end{lemma}
	
	\begin{proof}
		Denote $\overline{B} = Z^{-1}B$, $F = Z^{-1}Z$. The algebra $\overline{B}$ has a basis $e_1,\dots,e_n,[e_1],\dots,[e_n]$, where $e_i=(0,\dots,0,1,0,\dots,0)$ and the unit is at the $i$-th place. There exists $z\in Z$ such that $e_i = \frac{f_i}{z}$, $[e_i]=\frac{g_i}{z}$. Let $a=(\alpha_1,\dots,\alpha_n) + [(\beta_1,\dots,\beta_n)] \in B$, $\alpha_i,\beta_i\in F$. Let us introduce the notation
		\begin{gather*}
			c_n := ((f_na)f_n) = (0,\dots,0,\alpha_n)z^2,\\
			c_{i-1} := g_{i-1}(g_{i-1}c_i) = \alpha_n [e_{i-1}] z^{2(n-(i-2))} = \alpha_n g_{i-1} z^{2(n-(i-2)) - 1}
		\end{gather*}
		for $2\le i\le n$. Thus,
		\[(g_1+\dots+g_n)\sum\limits_{i=1}^n c_iz^{2(i-1)} = \alpha_n z^{2n}.\]
		Hence $\alpha_n z^{2n}\in B \cap F = Z$. Similarly, $\alpha_i z^{2n} \in B\cap F = Z$ for all $i\in\{1,\dots,n\}$. This means that
		\[az^{2n} \in \sum\limits_{i=1}^n (Ze_i + Z[e_i]).\]
		Thus, the $Z$-module $Bz^{2n}$ is embedded in a finitely generated $Z$-module. It remains to note that the $Z$-module $Bz^{2n}$ is isomorphic to the $Z$-module~$B$.
	\end{proof}
	
	\begin{lemma}
		Let $B$ be a right alternative superalgebra that is a central order in the superalgebra $B_{2|2}(\nu)$ and $Z = Z(B)_0$. Then $B$ is embedded in a finitely generated $Z$-module.
	\end{lemma}
	
	\begin{proof}
		Denote $\overline{B} = Z^{-1}B$, $F = Z^{-1}Z$. In the algebra $\overline{B}$ there is a basis $(1,0),(0,1),[(1,0)],[(0,1)]$ over the field $F$. There exists $z\in Z$ such that $(1,0) = \frac{b_1}{z}$, $(0,1) = \frac{b_2}{z}$, $[(1,0)] = \frac{d_1}{z}$, $[(0,1)] = \frac{d_2}{z}$ and $b_i,d_i\in B$. Let $a=(\alpha_1,\alpha_2) + [(\beta_1,\beta_2)] \in B$, $\alpha_i,\beta_i\in F$. Then
		\begin{gather*}
			(b_1 + b_2 + d_2)(b_1(ab_1)) = \alpha_1 z^3,\\
			(b_1 + b_2 + d_1)(b_2(ab_2)) = \alpha_2 z^3,\\
			(b_1 + b_2 + d_2)((b_1(ab_2))d_1) = \beta_1 z^4,\\
			(b_1 + b_2 + d_1)((b_2(ab_1))d_1) = \beta_2 z^4.
		\end{gather*}
		Thus, $\alpha_i z^3, \beta_i z^4 \in B\cap F = Z$, whence
		\[az^4 \in Z(1,0) + Z(0,1) + Z[(1,0)] + Z[(0,1)].\] Thus, the $Z$-module $Bz^4$ is embedded in the finitely generated $Z$-module. It remains to note that the $Z$-module $Bz^4$ is isomorphic to the $Z$-module~$B$.
	\end{proof}
	
	\medskip Next, consider central orders in asymmetric doubles. As is easy to see, the peculiarity of these algebras is that their even part is a central order in the algebra of second-order matrices.
	
	\begin{lemma}
		Let $B$ be a right alternative superalgebra which is a central order in the superalgebra $B_{4|4}(w)$ and $Z = Z(B)_0$. Then $B$ can be embedded into a finitely generated $Z$-module.
	\end{lemma}
	
	\begin{proof}
		Denote $\overline{B} = Z^{-1}B$, $F = Z^{-1}Z$. The algebra $\overline{B}$ has a basis $e_{ij},[e_{ij}]$, where $i,j\in\{1,2\}$. There is $z\in Z$ such that $e_{ij} = \frac{f_{ij}}{z}$, $[e_{ij}]=\frac{g_{ij}}{z}$, $\alpha = \frac{z_{\alpha}}{z}$, $\beta = \frac{z_{\beta}}{z}$, $\gamma = \frac{z_{\gamma}}{z}$, $\delta = \frac{z_{\delta}}{z}$. Let $a=\sum\alpha_{ij}e_{ij} + \sum\beta_{ij}[e_{ij}] \in B$, $\alpha_{ij},\beta_{ij}\in F$. Let
		\[w = \begin{pmatrix} \alpha & \beta \\
			\gamma & \delta
		\end{pmatrix}.\]
		Then
		\begin{gather*}
			f_{21}((f_{11}((f_{11}(af_{11}))f_{11}))f_{12}) + ((f_{11}((f_{11}(af_{11}))f_{11}))f_{12})f_{21} = \alpha_{11}z^6,\\
			f_{21}((((f_{12}(af_{12}))f_{22})f_{21})f_{12}) + ((((f_{12}(af_{12}))f_{22})f_{21})f_{12})f_{21} = \alpha_{12}z^6,\\
			f_{12}((((f_{21}(af_{21}))f_{11})f_{12})f_{21}) + ((((f_{21}(af_{21}))f_{11})f_{12})f_{21})f_{12} = \alpha_{21}z^6,\\
			((f_{22}((f_{22}(af_{22}))f_{22}))f_{21})f_{12} + f_{12}((f_{22}((f_{22}(af_{22}))f_{22}))f_{21}) = \alpha_{22}z^6,\\
			(z_{\alpha} + z_{\delta})(((g_{22}(af_{22}))f_{12})f_{21} + f_{21}((g_{22}(af_{22}))f_{12})) = \beta_{11}z^5,\\
			(z_{\alpha} + z_{\delta})(((g_{11}(af_{21}))f_{12})f_{21} + f_{21}((g_{11}(af_{21}))f_{12})) = -\beta_{12}z^5,\\
			(z_{\alpha} + z_{\delta})(((g_{22}(af_{12}))f_{21})f_{12} + f_{12}((g_{22}(af_{12}))f_{21})) = -\beta_{21}z^5,\\
			(z_{\alpha} + z_{\delta})(((g_{11}(af_{11}))f_{12})f_{21} + f_{21}((g_{11}(af_{11}))f_{12})) = \beta_{22}z^5.
		\end{gather*}
		Thus, $\alpha_{ij} z^6, \beta_{ij} z^5 \in B\cap F = Z$, whence
		\[az^6 \in \sum Ze_{ij} + \sum Z[e_{ij}].\] Thus, the $Z$-module $Bz^6$ is embedded in a finitely generated $Z$-module. It remains to note that the $Z$-module $Bz^6$ is isomorphic to the $Z$-module~$B$.
	\end{proof}
	Thus, we have proved the following theorem.
	\begin{theorem}
		Let $B$ be a right alternative superalgebra that is a central order in one of the superalgebras $B_{4|4}(w), B_{n|n}, B_{2|2}(\nu)$ and $Z=Z(B)_0$. Then $B$ is embedded in a finitely generated $Z$-module.
	\end{theorem}
	
	Moreover, the following corollary holds
	\begin{corollary}
		Let $A$ be an associative algebra over a field of characteristic not 2, which is a central order in the algebra of second-order matrices, and $B=A+M$ be a right alternative superalgebra, which is a central order in a simple finite-dimensional superalgebra. Then $B$ is embedded in a finitely generated $Z$-module, where $Z=Z(A)$.
	\end{corollary}
	
	\begin{proof}
		By the results of \cite{PS2}, the superalgebra $Z^{-1}B$ is either associative, or isomorphic to the Shestakov alternative superalgebra $S_{4|2}(\sigma)$, or isomorphic to the superalgebra $B_{4|4}(w)$. The first two cases were discussed in \cite{P2020}. The third case is discussed in Lemma~3.
	\end{proof}

	\end{document}